\def\marker{\>\hbox{${\vcenter{\vbox{
    \hrule height 0.4pt\hbox{\vrule width 0.4pt height 6pt
    \kern6pt\vrule width 0.4pt}\hrule height 0.4pt}}}$}\>}
\def\gpic#1{#1
     \medskip\par\noindent{\centerline{\box\graph}} \medskip}
\documentclass{article}[14]
\usepackage{verbatim}
\usepackage{amsthm}
\usepackage{fullpage}
\usepackage{subfig}
\usepackage{graphics}

\def\floor#1{\left\lfloor #1 \right\rfloor}

\newtheorem{lemma}{Lemma}
\newtheorem{theorem}{Theorem}
\def\cycl{\mathit{D}}
\def\cyc{\mathit{C}}
\def\H{\mathit{H}}
\begin{document}
\title{Multigraphs with $\Delta\geq 3$ are Totally-(2$\Delta$-1)-Choosable}
\author{Daniel W.\ Cranston\footnote{University of Illinois, Urbana-Champaign; currently at The Center for Discrete Mathematics and Theoretical Computer Science (DIMACS), Rutgers, \texttt{dcransto@dimacs.rutgers.edu}} 
}
\date{October 25, 2008}
\maketitle
\begin{abstract}
The \textit{total graph} $T(G)$ of a multigraph $G$ has as its vertices the set of edges and vertices of $G$ and has an edge between two vertices if their corresponding elements are either adjacent or incident in $G$. 
We show that if $G$ has maximum degree $\Delta(G)$, then $T(G)$ is $(2\Delta(G)-1)$-choosable.  
We give a linear-time algorithm that produces such a coloring.  
The best previous general upper bound for $\Delta(G) > 3$ was
$\floor{\frac32\Delta(G)+2}$, by Borodin et al.  
When $\Delta(G)=4$, our algorithm gives a better upper bound.  When $\Delta(G)\in\{3,5,6\}$, our algorithm matches the best known bound.
However, because our algorithm is significantly simpler, it runs in linear time (unlike the algorithm of Borodin et al.).
\end{abstract}

Throughout this paper, $G$ is a loopless multigraph. 
For convenience, we refer to edges and vertices as \textit{elements} of the graph.
The \textit{total graph} $T(G)$ of a graph $G$ has as its vertices the set of edges and vertices of $G$ and has an edge between two vertices if their corresponding elements are either adjacent or incident in $G$. 
Let $L$ be an assignment of lists to the vertices of a graph $G$.  
If $G$ has a proper coloring such that each vertex $v$ gets a color from its list $L(v)$, then we say that $G$ has an $L$-coloring.  
If $G$ always has an $L$-coloring when each vertex has a list of size $k$, then we say that $G$ is \textit{$k$-list-colorable} (or $k$-choosable).
In this paper, we study the problem of list-coloring a total graph.
If a total graph $T(G)$ is $k$-list-colorable, we say that $G$ is totally-$k$-list-colorable (or totally-$k$-choosable).
Often, our algorithm will greedily color all but a few edges and vertices of $G$; we generally call this uncolored subgraph $H$.  This motivates the following definition.
For a graph $G$ and a subgraph $H$, we use $G-H$ to mean $G-(V(H)\cup E(H))$ (thus, edge $uv$ may be present in $G-H$ even if one or both of vertices $u$ and $v$ are missing).
We say a graph algorithm runs in \textit{linear time} if for fixed maximum degree the algorithm runs in time linear in the number of vertices of the graph.  

Let $\Delta(G)$ denote the maximum vertex degree of a graph $G$.
Juvan et al.~\cite{juvan} showed that if $\Delta(G)=3$, then $G$ is totally-5-choosable.  
Skulrattanakulchai and Gabow~\cite{gabow} used these ideas to show that if $\Delta(G)=3$, 
then we can construct a total-5-list-coloring in linear time.  In this paper, we extend these ideas
further to show that if $\Delta(G)\geq 3$, then we can construct a total-($2\Delta(G)-1$)-list-coloring in linear time.
The best previous upper bound for $\Delta(G) > 3$ was $\floor{\frac32\Delta(G)+2}$, by Borodin et al.~\cite{borodin}.  
When $\Delta(G)=4$, our algorithm gives a new upper bound.  When $\Delta(G)\in\{3,5,6\}$, our algorithm matches the best known bound.
However, because our algorithm is significantly simpler, it runs in linear time (unlike the algorithm of Borodin et al.).
In Lemma~\ref{main-lemma}, we show how to greedily construct a total-$(2\Delta(G)-1)$-coloring for almost all of $G$.  The rest of this paper shows that we can extend the coloring to all of $G$. 

\begin{lemma}
\label{main-lemma}
Let $G$ be a connected multigraph.
If $G$ contains a vertex $v$ with $d(v)<\Delta(G)$, or $G$ contains an edge with multiplicity at least 3,
then $G$ is totally-($2\Delta(G)-1$)-choosable.
\end{lemma}
\begin{proof}
Let $k=\Delta(G)$.  
Suppose that $d(v) < k$ or that vertex $v$ is incident to edge $e$ with multiplicity at least 3.
The subdivision graph $S(G)$, is formed by replacing each edge of $G$ with a 
path of length 2.  
(The subdivision graph $S(G)$ has the same vertex set as the total graph $T(G)$, but it has fewer edges.  If we begin with $S(G)$ and add an edge between each pair of vertices at distance 2, we form $T(G)$.)
For each element $x\in V(G)\cup E(G)$, let $f(x)$ be the distance from $x$ to $v$ in $S(G)$.  
We will color all the elements of $G$ sequentially.
Let $N(x)$ be the set of vertices in $S(G)$ that are distance at most 2 from $x$; these are the elements that can restrict the color of $x$ in $T(G)$.
Note that we always have $|N(x)|\leq 2k$. 
Greedily color each element $x$ in decreasing order of $f(x)$.  
If $f(x)\geq 2$, let $u$ and $w$ be the first and second vertices after $x$ on a shortest path from $x$ to $w$ in $S(G)$.  
Since $f(w) < f(u) < f(x)$, vertices $u$ and $w$ will be uncolored at the time that we color $x$.  
Thus at most $2k-2$ colors are restricted from being used on $x$; hence we have a color available for $x$.

The elements $x$ with $f(x)=1$ are edges incident to $v$.  
In the case of a multiple edge, the last elements with distance 1 that we color are the copies of this edge, and finally $v$.
When at least one edge remains uncolored, at most $2k-2$ restrictions are imposed.
For the last edge, the multiplicity or the degree restriction on $v$ implies that at most $2k-3$ restrictions are imposed.
For $v$ also, in either case at most $2k-2$ restrictions are imposed, so we can complete the coloring.
\end{proof}

For any cycle $\cyc$, by Lemma~1 we can totally-($2\Delta(G)-1$)-color $G-E(\cyc)$.  
Our plan is to greedily total-color all of $G$ except for a few edges and vertices; we call these uncolored elements $\H$.
Lemmas~\ref{juvan1} and~\ref{juvan2} (from Juvan et al.~\cite{juvan}) 
show how to extend the coloring to $\H$ for various choices of $\H$.

For convenience, Juvan et al.\ define \textit{halfedges} to be edges with only one endpoint.  We use halfedge to describe an edge of $H$ which has only one endpoint in $H$.  A halfedge is colored similarly to an edge; the only difference is that a halfedge has only one endpoint in $H$, so it has at most $\Delta(G)-1$ incident edges in $H$.

\begin{lemma}(\cite{juvan})
\label{juvan1}
Let $H$ be a cycle with a halfedge attached to each vertex.  
If $L$ is a list assignment for $H$ such that 

$$
|L(t)|\geq \left\{
\begin{array}{ll}
5, & \mbox{if $t$ is a full edge,} \\
4, & \mbox{if $t$ is a vertex,} \\
2, & \mbox{if $t$ is a halfedge,} 
\end{array}
\right.
$$
then $H$ admits an $L$-total-coloring.
\end{lemma}

If $G$ has girth at least 5, we will show (in Lemma~\ref{structural}) that $G$ contains an induced cycle $\cyc$ with a matching incident to the vertices of $\cyc$.  
In this case, we will greedily color all the elements of $G$ except for $\cyc$ and the incident matching.  
By treating the edges of the matching as halfedges, 
we use Lemma~\ref{juvan1} to finish the coloring (we give the details in Theorem~\ref{main}).
The next two lemmas consider the cases when $G$ has girth at most 4.  
In each case we find a small subgraph $H$ and greedily total-color $G-H$; in Lemmas~\ref{juvan2} and~\ref{small} we show that 
we are able to extend the coloring to $H$.

In Lemma~\ref{juvan2} we refer to \textit{thick} halfedges and \textit{thin} halfedges.  Both are halfedges as described above; the only difference is that thick halfedges will receive lists of size 3, whereas thin halfedges will receive lists of size 2.  Thick halfedges always appear in pairs; they designate nonincident halfedges in $H$ that correspond to incident edges in the larger graph.
%
\vspace{-.3in}
\begin{figure}[h!bt]
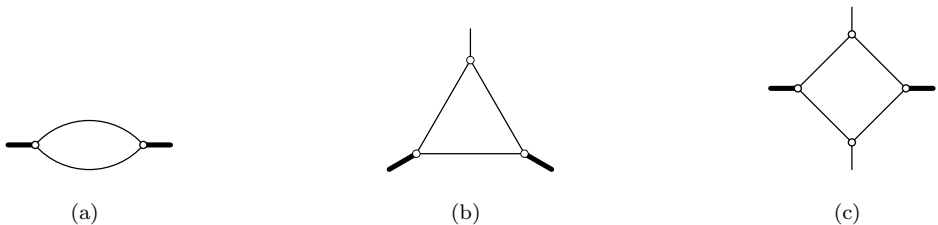

\subfloat[]{
\begin{minipage}[b]{0.30\linewidth}
\centering
\gpic{
\expandafter\ifx\csname graph\endcsname\relax \csname newbox\endcsname\graph\fi
\expandafter\ifx\csname graphtemp\endcsname\relax \csname newdimen\endcsname\graphtemp\fi
\setbox\graph=\vtop{\vskip 0pt\hbox{%
    \special{pn 8}%
    \special{ar 425 -121 378 378 0.722734 2.418858}%
    \special{ar 425 378 378 378 -2.418858 -0.722734}%
    \special{pn 28}%
    \special{pa 142 128}%
    \special{pa 0 128}%
    \special{fp}%
    \special{pa 708 128}%
    \special{pa 850 128}%
    \special{fp}%
    \special{sh 1.000}%
    \special{ia 142 128 22 22 0 6.28319}%
    \special{sh 0.000}%
    \special{ia 142 128 15 15 0 6.28319}%
    \special{sh 1.000}%
    \special{ia 708 128 22 22 0 6.28319}%
    \special{sh 0.000}%
    \special{ia 708 128 15 15 0 6.28319}%
    \hbox{\vrule depth0.256in width0pt height 0pt}%
    \kern 0.850in
  }%
}%
}
\end{minipage}}%
\subfloat[]{
\begin{minipage}[b]{0.30\linewidth}
\centering
\gpic{
\expandafter\ifx\csname graph\endcsname\relax \csname newbox\endcsname\graph\fi
\expandafter\ifx\csname graphtemp\endcsname\relax \csname newdimen\endcsname\graphtemp\fi
\setbox\graph=\vtop{\vskip 0pt\hbox{%
    \special{pn 8}%
    \special{pa 425 164}%
    \special{pa 708 654}%
    \special{pa 142 654}%
    \special{pa 425 164}%
    \special{fp}%
    \special{pa 425 164}%
    \special{pa 425 0}%
    \special{fp}%
    \special{pn 28}%
    \special{pa 708 654}%
    \special{pa 850 736}%
    \special{fp}%
    \special{pa 142 654}%
    \special{pa 0 736}%
    \special{fp}%
    \special{sh 1.000}%
    \special{ia 425 164 22 22 0 6.28319}%
    \special{sh 0.000}%
    \special{ia 425 164 17 17 0 6.28319}%
    \special{sh 1.000}%
    \special{ia 708 654 22 22 0 6.28319}%
    \special{sh 0.000}%
    \special{ia 708 654 17 17 0 6.28319}%
    \special{sh 1.000}%
    \special{ia 142 654 22 22 0 6.28319}%
    \special{sh 0.000}%
    \special{ia 142 654 17 17 0 6.28319}%
    \hbox{\vrule depth0.736in width0pt height 0pt}%
    \kern 0.850in
  }%
}%
}
\end{minipage}}%
\subfloat[]{
\begin{minipage}[b]{0.30\linewidth}
\centering
\gpic{
\expandafter\ifx\csname graph\endcsname\relax \csname newbox\endcsname\graph\fi
\expandafter\ifx\csname graphtemp\endcsname\relax \csname newdimen\endcsname\graphtemp\fi
\setbox\graph=\vtop{\vskip 0pt\hbox{%
    \special{pn 8}%
    \special{pa 425 142}%
    \special{pa 708 425}%
    \special{pa 425 708}%
    \special{pa 142 425}%
    \special{pa 425 142}%
    \special{fp}%
    \special{pa 425 142}%
    \special{pa 425 0}%
    \special{fp}%
    \special{pn 28}%
    \special{pa 708 425}%
    \special{pa 850 425}%
    \special{fp}%
    \special{pn 8}%
    \special{pa 425 708}%
    \special{pa 425 850}%
    \special{fp}%
    \special{pn 28}%
    \special{pa 142 425}%
    \special{pa 0 425}%
    \special{fp}%
    \special{sh 1.000}%
    \special{ia 425 142 22 22 0 6.28319}%
    \special{sh 0.000}%
    \special{ia 425 142 15 15 0 6.28319}%
    \special{sh 1.000}%
    \special{ia 708 425 22 22 0 6.28319}%
    \special{sh 0.000}%
    \special{ia 708 425 15 15 0 6.28319}%
    \special{sh 1.000}%
    \special{ia 425 708 22 22 0 6.28319}%
    \special{sh 0.000}%
    \special{ia 425 708 15 15 0 6.28319}%
    \special{sh 1.000}%
    \special{ia 142 425 22 22 0 6.28319}%
    \special{sh 0.000}%
    \special{ia 142 425 15 15 0 6.28319}%
    \hbox{\vrule depth0.850in width0pt height 0pt}%
    \kern 0.850in
  }%
}%
}
\end{minipage}}%
\caption{(a) A double edge with each endpoint incident to a thick halfedge.  
(b) A 3-cycle with two vertices incident to thick halfedges and the third vertex incident to a thin halfedge.
(c) A 4-cycle with two nonadjacent vertices incident to thick halfedges and the other two vertices incident to thin halfedges.
}
\end{figure}
%
\begin{lemma}(\cite{juvan})
\label{juvan2}
Let $H$ be isomorphic to one of the multigraphs in Figure 1.  If $L$ is a list assignment for $H$ such that 
\begin{eqnarray*}
|L(t)|\geq \left\{
\begin{array}{ll}
5, & \mbox{if $t$ is a proper edge,} \\
4, & \mbox{if $t$ is a vertex,} \\
3, & \mbox{if $t$ is a thick halfedge} \\
2, & \mbox{if $t$ is a thin halfedge,} 
\end{array}
\right.
\end{eqnarray*}
then $H$ admits an $L$-total-coloring such that the two thick halfedges receive distinct colors.
\end{lemma}

Juvan et al.\ were the first to show that a graph $G$ with $\Delta(G)=3$
is totally-5-choosable.  Most of the subgraphs $H$ that we encounter in our proof also arose in their proof (as seen in Lemmas~\ref{juvan1} and~\ref{juvan2}).  However, there are a few additional subgraphs that we must deal with as we prove our generalization of their result.  We handle these subgraphs in the following lemma.  

\vspace{-.1in}
\begin{figure}[h!bt]
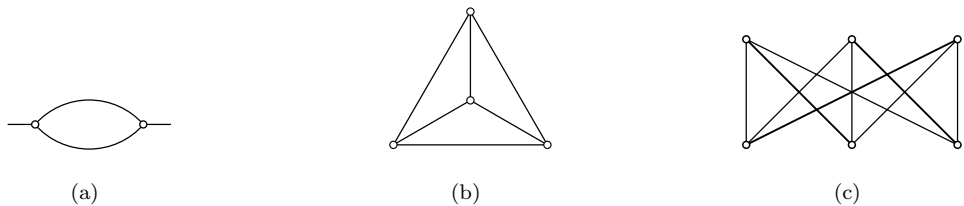

\subfloat[]{
\begin{minipage}[b]{0.30\linewidth}
\centering
\gpic{
\expandafter\ifx\csname graph\endcsname\relax \csname newbox\endcsname\graph\fi
\expandafter\ifx\csname graphtemp\endcsname\relax \csname newdimen\endcsname\graphtemp\fi
\setbox\graph=\vtop{\vskip 0pt\hbox{%
    \special{pn 8}%
    \special{ar 425 -121 378 378 0.722734 2.418858}%
    \special{ar 425 378 378 378 -2.418858 -0.722734}%
    \special{pa 142 128}%
    \special{pa 0 128}%
    \special{fp}%
    \special{pa 708 128}%
    \special{pa 850 128}%
    \special{fp}%
    \special{sh 1.000}%
    \special{ia 142 128 22 22 0 6.28319}%
    \special{sh 0.000}%
    \special{ia 142 128 15 15 0 6.28319}%
    \special{sh 1.000}%
    \special{ia 708 128 22 22 0 6.28319}%
    \special{sh 0.000}%
    \special{ia 708 128 15 15 0 6.28319}%
    \hbox{\vrule depth0.256in width0pt height 0pt}%
    \kern 0.850in
  }%
}%
}
\end{minipage}}%
\subfloat[]{
\begin{minipage}[b]{0.30\linewidth}
\centering
\gpic{
\expandafter\ifx\csname graph\endcsname\relax \csname newbox\endcsname\graph\fi
\expandafter\ifx\csname graphtemp\endcsname\relax \csname newdimen\endcsname\graphtemp\fi
\setbox\graph=\vtop{\vskip 0pt\hbox{%
    \special{pn 8}%
    \special{pa 425 22}%
    \special{pa 828 720}%
    \special{pa 22 720}%
    \special{pa 425 22}%
    \special{fp}%
    \special{pa 425 22}%
    \special{pa 425 487}%
    \special{fp}%
    \special{pa 828 720}%
    \special{pa 425 487}%
    \special{fp}%
    \special{pa 22 720}%
    \special{pa 425 487}%
    \special{fp}%
    \special{sh 1.000}%
    \special{ia 425 487 22 22 0 6.28319}%
    \special{sh 0.000}%
    \special{ia 425 487 16 16 0 6.28319}%
    \special{sh 1.000}%
    \special{ia 425 22 22 22 0 6.28319}%
    \special{sh 0.000}%
    \special{ia 425 22 16 16 0 6.28319}%
    \special{sh 1.000}%
    \special{ia 828 720 22 22 0 6.28319}%
    \special{sh 0.000}%
    \special{ia 828 720 16 16 0 6.28319}%
    \special{sh 1.000}%
    \special{ia 22 720 22 22 0 6.28319}%
    \special{sh 0.000}%
    \special{ia 22 720 16 16 0 6.28319}%
    \hbox{\vrule depth0.742in width0pt height 0pt}%
    \kern 0.850in
  }%
}%
}
\end{minipage}}%
\subfloat[]{
\begin{minipage}[b]{0.30\linewidth}
\centering
\gpic{
\expandafter\ifx\csname graph\endcsname\relax \csname newbox\endcsname\graph\fi
\expandafter\ifx\csname graphtemp\endcsname\relax \csname newdimen\endcsname\graphtemp\fi
\setbox\graph=\vtop{\vskip 0pt\hbox{%
    \special{pn 8}%
    \special{pa 22 22}%
    \special{pa 22 575}%
    \special{pa 575 22}%
    \special{pa 575 575}%
    \special{pa 1128 22}%
    \special{pa 1128 575}%
    \special{pa 22 22}%
    \special{fp}%
    \special{pn 11}%
    \special{pa 22 22}%
    \special{pa 575 575}%
    \special{fp}%
    \special{pa 575 22}%
    \special{pa 1128 575}%
    \special{fp}%
    \special{pa 1128 22}%
    \special{pa 22 575}%
    \special{fp}%
    \special{sh 1.000}%
    \special{ia 22 22 22 22 0 6.28319}%
    \special{sh 0.000}%
    \special{ia 22 22 14 14 0 6.28319}%
    \special{sh 1.000}%
    \special{ia 575 22 22 22 0 6.28319}%
    \special{sh 0.000}%
    \special{ia 575 22 14 14 0 6.28319}%
    \special{sh 1.000}%
    \special{ia 1128 22 22 22 0 6.28319}%
    \special{sh 0.000}%
    \special{ia 1128 22 14 14 0 6.28319}%
    \special{sh 1.000}%
    \special{ia 22 575 22 22 0 6.28319}%
    \special{sh 0.000}%
    \special{ia 22 575 14 14 0 6.28319}%
    \special{sh 1.000}%
    \special{ia 575 575 22 22 0 6.28319}%
    \special{sh 0.000}%
    \special{ia 575 575 14 14 0 6.28319}%
    \special{sh 1.000}%
    \special{ia 1128 575 22 22 0 6.28319}%
    \special{sh 0.000}%
    \special{ia 1128 575 14 14 0 6.28319}%
    \hbox{\vrule depth0.597in width0pt height 0pt}%
    \kern 1.150in
  }%
}%
}
\end{minipage}}%
\caption{(a) A double edge with each endpoint incident to a thin halfedge.  
(b) A complete graph on 4 vertices.
(c) A complete bipartite graph with each part of size 3.
}
\end{figure}
%

\begin{lemma}
\label{small}
Let $H$ be isomorphic to one of the multigraphs in Figure~2.  If $L$ is a list
assignment for $H$ such that 
\begin{eqnarray}
|L(t)|\geq \left\{
\begin{array}{ll}
5, & \mbox{if $t$ is a proper edge,} \\
4, & \mbox{if $t$ is a vertex,} \\
2, & \mbox{if $t$ is a halfedge,} 
\end{array}
\right.
\end{eqnarray}
Then $H$ admits an $L$-total-coloring.
\end{lemma}
\begin{proof}

If $H$ is congruent to the multigraph in Figure 2a, then
let $e_1$ and $e_2$ be the parallel edges.  Let the halfedges be $e_3$ (incident to vertex $v_1$) and $e_4$ (incident to vertex $v_2$).  
We may assume that there are equalities in (1); otherwise we discard colors from the lists.

Since $|L(v_1)|+|L(e_4)|>|L(e_1)|$, either $v_1$ and $e_4$ have a common color $c$ available, or there exists
a color $c\in (L(v_1)\cup L(e_4))\setminus L(e_1)$.  If color $c$ is available on both $v_1$ and $e_4$, we use it on both elements;
otherwise use color $c$ on the element where it is available (either $v_1$ or $e_4$) and color the other element arbitrarily.  
Note that $L(e_1)$ has 4 elements other than the colors used on $v_1$ and $e_4$.  Hence, after greedily coloring $e_3$, $v_2$, and $e_2$
in order, a color remains available for $e_1$.

If $H\cong K_4$ (as shown in Figure 2b), then greedily color the vertices of $H$ in some order.  
Note that each edge now has at least 3 colors available (since each edge lost at most one color to each endpoint); let $L'(e)$ denote the list of remaining available colors on each edge $e$ after coloring the vertices of $H$.
Suppose that we cannot give distinct colors to the edges.
By Hall's Theorem there exists a set $E_1$ of edges such that $|\cup_{e\in E_1}L'(e)|<|E_1|$.
Set $E_1$ must contain at least 4 edges, since each edge has at least 3 colors available.
Among any 4 edges of $K_4$ there are two nonincident edges, call them $e_1$ and $e_2$. 
Since $|L(e_1)\cup L(e_2)| < |E_1|\leq 6$, edges $e_1$ and $e_2$ must have a common available color, $c$.  
Use color $c$ on edges $e_1$ and $e_2$.
The four remaining uncolored edges form a 4-cycle.
Note that each uncolored edge has at least 2 colors available.
Erd\"{o}s, Rubin, and Taylor~\cite{ERT} proved that cycles of even length are
edge-2-choosable, so we can finish the coloring.

If $H\cong K_{3,3}$ (as shown in Figure 3b), 
then greedily color the vertices of $H$.  
Note that each edge now has at least 3 colors available.  
Galvin~\cite{galvin} proved that a bipartite multigraph $H$ is edge-$\Delta(H)$-choosable.
Hence, we can finish the coloring.
\end{proof}

Our final lemma is structural.  We use it to show that every multigraph 
$G$ contains a subgraph $H$ for which we can extend a total-$(2\Delta(G)-1)$-list-coloring of $G-H$ to such a coloring of $G$.
For convenience we consider a double-edge to be a cycle of length 2.

\begin{lemma}
\label{structural}
If $G$ is a regular multigraph,
then $G$ contains an induced cycle $\cyc$ with the following property.  
If the cycle $\cyc$ contains any pair of vertices with a common neighbor not on $\cyc$, then $|V(\cyc)|\leq 4$.  
Furthermore, we can find such a cycle in linear time.
\end{lemma}
\begin{proof}
If $G$ has a multiple edge, then $G$ has a cycle of size 2.  We can find it in linear time.  So we assume $G$ is simple.

Choose an arbitrary vertex $v$.  Using breadth-first-search, find a shortest cycle $\cycl$ through $v$.
If there exist vertices $w$ and $x$ on $\cycl$ with a common neighbor $y$ not on $\cycl$, then either $w$ and $x$ are
adjacent or $w$ and $x$ have a common neighbor $z$ on $\cycl$ (otherwise we could find a shorter cycle through $v$).  
In either case let $\cyc$ be the 3-cycle or 4-cycle containing $w$, $x$, and $y$.
If $\cycl$ does not contain such a pair $w,x$, then let $\cyc=\cycl$.
\end{proof}

By combining Lemmas~\ref{main-lemma} through~\ref{structural}, we prove our main result.
\begin{theorem}
\label{main}
If $G$ is a multigraph with maximum degree $\Delta(G)$, then $G$ is totally-($2\Delta(G)-1$)-choosable.
Furthermore, given lists of size $2\Delta(G)-1$, we can construct a coloring in linear time.
\end{theorem}
\begin{proof}
If $G$ is disconnected, then we color each component separately.
If $G$ is not $\Delta(G)$-regular or 
if $G$ contains an edge with multiplicity at least 3, then by Lemma~\ref{main-lemma} we can color $G$.

If $G$ contains an edge $uv$ with multiplicity 2, then 
let $e_1$ and $e_2$ be additional edges incident to $u$ and $v$, respectively.
We view $e_1$ and $e_2$ as halfedges (thick if they have a common endpoint, and thin otherwise).
This subgraph (call it $H$) is isomorphic to either Figure~1a or Figure~2a.

By Lemma~\ref{main-lemma}, we can greedily color $G-E(H)$.  
We will use either Lemma~\ref{juvan2} or Lemma~\ref{small} to complete the coloring.
To apply Lemma~\ref{juvan2} or Lemma~\ref{small}, we must verify that each uncolored vertex, edge, and halfedge has a sufficient number of available colors.
We do this as follows.

Let $k=\Delta(G)$.
An uncolored vertex ($u$ or $v$) is incident to at most $k-3$ colored edges and at most $k-2$ colored vertices.  So it has at least $(2k-1)-(k-3)-(k-2)=4$ available colors.
An uncolored edge (one of the $uv$ edges) is not incident to any colored vertices; each endpoint is incident to at most $k-3$ colored edges.  So an edge has at least $(2k-1)-2(k-3)=5$ available colors.  A thin halfedge ($e_1$ or $e_2$) is incident to a single colored vertex; one endpoint is incident to at most $k-1$ colored edges and the other endpoint is incident to at most $k-3$ colored edges.  So a thin halfedge has at least $(2k-1) - 1 - (k-1)-(k-3)=2$ available colors.  A thick halfedge has one additional available color, since it is incident to the other thick halfedge, which is uncolored.
Hence, we can assume that $G$ is a regular simple graph.

Find a cycle $\cyc$ as described in Lemma~\ref{structural}.  
By Lemma~\ref{main-lemma}, greedily color $G-E(\cyc)$.
If $|V(\cyc)|>4$, then uncolor the vertices of $\cyc$ and uncolor one edge incident to each vertex of $\cyc$.  
We treat each uncolored incident edge as a halfedge, then finish the coloring by Lemma~\ref{juvan1}.
To apply Lemma~\ref{juvan1}, we must verify that the lists of colors are big enough.
However, the analysis is similar to that given above, so we omit it.
Thus, if $|V(\cyc)|>4$, then Lemma~\ref{juvan1} applies and we can finish the coloring.  So we may assume that $|V(\cyc)|\leq 4$.

Suppose that $|V(\cyc)|=3$.  Uncolor an edge incident to each vertex of $\cyc$.  
If the three uncolored incident edges have a common endpoint, 
then also uncolor this common endpoint.
The uncolored subgraph is isomorphic to $K_4$.
By Lemma~\ref{small}, we can finish the coloring.
If the uncolored subgraph is not isomorphic to $K_4$, then it is isomorphic to
a subgraph in Lemma~\ref{juvan1} or Lemma~\ref{juvan2} (Figure 1b).  In each case, we can finish the coloring.  To apply Lemma~\ref{juvan1},~\ref{juvan2}, or~\ref{small} we must verify
that the lists of available colors are big enough.  
Again, we omit the case analysis.

Finally, suppose $|V(\cyc)|=4$.  If two adjecent vertices of $\cyc$ have a common neighbor not on $\cyc$ then instead we let
$\cyc$ be this 3-cycle and we are in the above case.  Otherwise, uncolor an edge incident to each vertex of $\cyc$.  If at most one
pair of these incident edges has a common endpoint, then we can finish the coloring by Lemma~\ref{juvan1} or Lemma~\ref{juvan2} (Figure 1c).
If two pairs of these incident edges have common endpoints, call these endpoints $u$ and $v$.  If $u$ and $v$ are adjacent,
then we have a subgraph isomorphic to $K_{3,3}$, so by Lemma~\ref{small} (Figure 2c) we can finish the coloring.  If $u$ and $v$ are not adjacent, 
then instead replace $\cyc$ with the cycle induced by $(V(\cyc)\cup u)\backslash\{w\}$, 
where $w$ is a vertex of $\cyc$ but $w$ is not adjacent to $u$.
Because edge $uv$ is not present in $G$, we can choose one edge incident to each vertex of this new 4-cycle so that at most one pair of incident edges has a common endpoint.
Now we are in the previous case, and we can finish by Lemma~\ref{juvan1} or Lemma~\ref{juvan2} (Figure 1c).
\end{proof}

\section*{Acknowledgements}
Thanks to Douglas Woodall for bringing this problem to my attention and for his insights that greatly improved the exposition of this paper.
Thanks to Doug West and two referees, whose comments each improved the exposition.
Thanks to Kevin Milans for helpful conversation.

\end{document}